\documentclass{article}

\usepackage{a4wide}

\usepackage{amsthm,amsfonts,amssymb,amsmath,cite}

\newtheorem{theorem}{Theorem}
\newtheorem{corollary}[theorem]{Corollary}
\newtheorem{lemma}[theorem]{Lemma}
\newtheorem{definition}[theorem]{Definition}
\newtheorem{example}[theorem]{Example}
\newtheorem{remark}[theorem]{Remark}
\newtheorem{problem}[theorem]{Problem}
\newtheorem*{notation}{Notation}

\begin{document}

\title{Hight-order Noether's Theorem for Nonsmooth Extremals\\
of Isoperimetric Variational Problems with Time Delay}

\author{G. S. F. Frederico$^{1,2}$ and M. J. Lazo$^3$\\
\texttt{gastao.frederico@ua.pt} \texttt{}}

\date{\mbox{$^1$Department of Mathematics,}\\
Federal University of Santa Catarina,
Florianopolis, SC, Brazil\\[0.3cm]
    $^2$Department of Science and Technology,\\
University of Cape Verde, Praia, Santiago, Cape Verde\\[0.3cm]
$^3$Institute of Mathematics, Statistics and Physics,
Federal University of Rio Grande, Rio Grande, RS, Brazil\\}

\maketitle


\begin{abstract}
We obtain a nonsmooth higher-order extension of Noether's symmetry
theorem for variational isoperimetric problems with delayed
arguments. The result is proved to be valid in the class of
Lipschitz functions, as long as the delayed higher-order
Euler--Lagrange extremals are restricted to those that satisfy the
delayed higher-order DuBois--Reymond necessary optimality condition.
The important case of delayed isoperimetric optimal control problems
is considered as well.

\medskip


\medskip


\end{abstract}


\section{Introduction}

Noether's theorem, published in 1918 \cite{Noether:1918} ina seminal work, is a
central result of the calculus of variations that explains all
physical conservation laws based upon the action principle. It is a very general
result, asserting that ``to every variational symmetry of the
problem there corresponds a conservation law''. Noether's principle
gives powerful insights from the various transformations that make a
system invariant. For instance, in mechanics the invariance of a
physical system with respect to spatial translation gives
conservation of linear momentum; invariance with respect to rotation
gives conservation of angular momentum; and invariance with respect
to time translation gives conservation of energy \cite{MR2761345}.
As a consequence, the Noether's theorem is usually considered 
the most important mathematical theorem of the twentieth century for Physics.
On the other hand, the calculus of variations is now part of a more vast discipline,
called optimal control \cite{CD:MR29:3316b}, and Noether's principle
still holds in this more general setting \cite{ejc}.
 Within the years, this result has been studied by many 
 authors and generalized in different directions: see
\cite{Bartos,Gastao:PhD:thesis,GF:JMAA:07,GF2012,MR2323264,book:frac,NataliaNoether,ejc}
and references therein. In particular, in the recent paper
\cite{GF2013,{GFML2016}}, Noether's theorem was formulated for
variational problems with delayed arguments. The result is important
because problems with delays play a crucial role in the modeling of
real-life phenomena in various fields of applications \cite{ChiLoi,EFridman,GoKeMa}.
In order to prove Noether's theorem with delays, it was assumed that
admissible functions are $\mathcal{C}^2$-smooth and that Noether's
conserved quantity holds along all $\mathcal{C}^2$-extremals of the
Euler--Lagrange equations with time delay \cite{GF2012}. Here we
show that to formulate higher-order Noether's theorem with time
delays for nonsmooth functions, it is enough to restrict the set of
delayed isoperimetric higher-order Euler--Lagrange extremals to
those that satisfy the delayed isoperimetric higher-order
DuBois--Reymond condition. Moreover, we prove that this result can
be generalized to more general isoperimetric optimal control
problemas.

The text is organized as follows. In Section~\ref{sec:prelim} we
give a short review of the results for the fundamental isoperimetric
problem of variational calculus with delayed arguments. The main
contributions of the paper appear in Sections~\ref{sec:MRHO} and
\ref{sec:MRHO1}: we prove an Euler--Lagrange and DuBois--Reymond
optimality type conditions for nonsmooth higher-order isoperimetric
variational problems with delayed arguments
(Theorems~\ref{Thm:ELdeordm} and~\ref{theo:cDRifm}, respectively),
isoperimetric higher-order Noether symmetry theorem with time delay
for Lipschitz functions (Theorem~\ref{thm:Noether}) and a delayed
Noether's symmetry theorem (Theorem~\ref{thm:NT:OC}) for
isoperimetric optimal control problems. Two examples of application
are given in Section~\ref{exe}.


\section{Preliminaries}
\label{sec:prelim}

In this section we review necessary results on the calculus of
variations with time delay. For more on variational problems with
delayed arguments we refer the reader to
\cite{Basin:book,Bok,GoKeMa,DH:1968,Kra,Kharatishvili,Ros}.

We begin by defining the isoperimetric variational problem as in
\cite{GFML2016}.

\begin{problem}(The isoperimetric variational problem with time delay)
\label{Pb1} The isoperimetric  problem of the calculus of variations
consists of minimizing a functional
\begin{equation}
\label{Pe} J^{\tau}[q(\cdot)] = \int_{t_{1}}^{t_{2}}
L\left(t,q(t),\dot{q}(t),q(t-\tau),\dot{q}(t-\tau)\right) dt\,
\end{equation}
 subject to the isoperimetric equality constraints
\begin{equation}
\label{CT} I^{\tau}[q(\cdot)]=\int_{t_1}^{t_2}
g\left(t,q(t),\dot{q}(t),q(t-\tau),\dot{q}(t-\tau)\right)
dt=l,\,\,\,l\in\mathbb{R}^k\,,
\end{equation}
and boundary conditions \begin{equation} \label{Pe2}
q(t)=\delta(t)~\textnormal{ for
}~t\in[t_{1}-\tau,t_{1}]~\textnormal{ and }~q(t_2)=q_{t_2}.
\end{equation}.

We assume that $L,g :[t_1,t_2] \times \mathbb{R}^{4} \rightarrow
\mathbb{R}$, are a $\mathcal{C}^{2}$-functions
with respect to all their arguments, admissible functions $q(\cdot)$
are $\mathcal{C}^2$-smooth, $t_{1}< t_{2}$ are fixed in
$\mathbb{R}$, $\tau$ is a given positive real number such that
$\tau<t_{2}-t_{1}$, $l$ is a specified real constant and $\delta$ is
a given piecewise smooth function on $[t_1-\tau,t_1]$.
\end{problem}

 Throughout the text,
$\partial_{i}L$ denotes the partial derivative of $L$ with respect
to its $i$th argument. For convenience of notation, we introduce the
operator $[\cdot]_{\tau}$ defined by
$$
[q]_{\tau}(t)=(t,q(t),\dot{q}(t),q(t-\tau),\dot{q}(t-\tau)).
$$

The arguments of the calculus of variations assert that, by using the
Lagrange multiplier rule, Problem~\ref{Pb1} is equivalent to the
following augmented problem \cite[$\S12.1$]{CD:Gel:1963}: to
minimize
\begin{equation}
\label{agp}
\begin{split}
J^{\tau}[q(\cdot),\lambda]
&= \int_{t_1}^{t_2} F[q,\lambda]_{\tau}(t)dt\\
&:=\int_{t_1}^{t_2} \left[L[q]_{\tau}(t) -\lambda \cdot
g[q]_{\tau}(t)\right] dt
\end{split}
\end{equation}
subject to \eqref{Pe2}, where
$[q,\lambda]_{\tau}(t)=(t,q(t),\dot{q}(t),q(t-\tau),\dot{q}(t-\tau),\lambda)$.

 The augmented Lagrangian
\begin{equation}
\label{eq:aug:Lag} F:=L-\lambda \cdot g,\,\lambda\in\mathbb{R}
\end{equation} has an important role in our study.

The notion of extremizer (a local minimizer or a local maximizer)
can to be found in \cite{CD:Gel:1963}. Extremizers can be classified
as normal or abnormal.

\begin{definition}\label{def:extr}
An extremizer of Problem~\ref{Pb1} that does not satisfy the
Euler--Lagrange equations
\begin{equation}
\label{EL11}
\begin{cases}
\frac{d}{dt}\left\{\partial_{3}g[q]_{\tau}(t)+
\partial_{5}g[q]_{\tau}(t+\tau)\right\}
=\partial_{2}g[q]_{\tau}(t)+\partial_{4}g[q]_{\tau}(t+\tau),
\quad t_{1}\leq t\leq t_{2}-\tau,\\
\frac{d}{dt}\partial_{3}g[q]_{\tau}(t) =\partial_{2}g[q]_{\tau}(t),
\quad t_{2}-\tau\leq t\leq t_{2}.
\end{cases}
\end{equation}
is said to be a normal extremizer; otherwise (i.e., if it satisfies
\eqref{EL11} for all $t\in[t_1,t_2]$), is said to be abnormal.
\end{definition}

The following theorem gives a necessary condition for $q(\cdot)$ to
be a solution of  Problem~\ref{Pb1} under the assumption that
$q(\cdot)$ is a normal extremizer.

\begin{theorem}
\label{Thm:FractELeq1} If $q(\cdot)\in
C^2\left([t_{1}-\tau,t_{2}]\right)$ is a normal
extremizer to Problem~\ref{Pb1}, then it satisfies the following
\emph{isoperimetric Euler--Lagrange equation with time delay}:
\begin{equation}
\label{eq:eldf11}
\begin{cases}
\frac{d}{dt}\left\{\partial_{3}F[q,\lambda]_{\tau}(t)+
\partial_{5}F[q,\lambda]_{\tau}(t+\tau)\right\}\\
=\partial_{2}F[q,\lambda]_{\tau}(t)+\partial_{4}F[q,\lambda]_{\tau}(t+\tau),
\quad t_{1}\leq t\leq t_{2}-\tau,\\
\frac{d}{dt}\partial_{3}F[q,\lambda]_{\tau}(t)
=\partial_{2}F[q,\lambda]_{\tau}(t), \quad t_{2}-\tau\leq t\leq
t_{2}\,,
\end{cases}
\end{equation}
$t \in [t_1,t_2]$, where $F$ is the augmented Lagrangian
\eqref{eq:aug:Lag} associated with Problem~\ref{Pb1}.
\end{theorem}

\begin{remark}
\label{re:EL} If one extends the set of admissible functions in
problem \eqref{Pe}--\eqref{Pe2} to the class of Lipschitz continuous
functions, then the isoperimetric Euler--Lagrange equations
\eqref{eq:eldf11} remain valid. This result is obtained from our
Corollary~\ref{cor:16} by choosing $m = 1$.
\end{remark}

\begin{definition}[Isoperimetrc extremals with time delay]
\label{def:scale:ext}
The solutions $q(\cdot)\in
C^2\left([t_{1}-\tau,t_{2}]\right)$ of the
Euler--Lagrange equations (\ref{eq:eldf11}) are called
\emph{isoperimetric extremals with time delay}.
\end{definition}

\begin{theorem}[Isoperimetric DuBois--Reymond necessary condition with time delay \cite{GFML2016}]
\label{theo:cdrnd} If $q(\cdot)$ is an isoperimetric extremals with
time delay such that
\begin{equation}\label{CDUR}
\partial_4F[q]_{\tau}(t+\tau)\cdot \dot{q}(t)+\partial_5F[q]_{\tau}(t+\tau)\cdot
\ddot{q}(t)=0
\end{equation}

for all $t\in[t_1-\tau,t_2-\tau]$, then it satisfies the following
conditions:
\begin{equation}
\label{eq:cdrnd}
\frac{d}{dt}\left\{F[q]_{\tau}(t)-\dot{q}(t)\cdot(\partial_{3}
F[q]_{\tau}(t) +\partial_{5} F[q]_{\tau}(t+\tau))\right\} =
\partial_{1} F[q]_{\tau}(t)
\end{equation}
for $t_1\leq t\leq t_{2}-\tau$, and
\begin{equation}
\label{eq:cdrnd1} \frac{d}{dt}\left\{F[q]_{\tau}(t)
-\dot{q}(t)\cdot\partial_{3} F[q]_{\tau}(t)\right\} =\partial_{1}
F[q]_{\tau}(t)
\end{equation}
for $t_2-\tau\leq t\leq t_{2}$, where $F$ is defined in
\eqref{eq:aug:Lag}.
\end{theorem}

\begin{remark}
If we assume that admissible functions in problem
\eqref{Pe}--\eqref{Pe2} are Lipschitz continuous, then one can show
that the DuBois--Reymond necessary conditions with time delay
\eqref{eq:cdrnd} are still valid (cf. Corollary~\ref{cor:DR:m1} by
choosing $m = 1)$.
\end{remark}

\begin{definition}[Invariance up to a gauge-term]
\label{def:invndLIP} Consider the following $s$-parameter group of
infinitesimal transformations:
\begin{equation}
\label{eq:tinf}
\begin{cases}
\bar{t} = t + s\eta(t,q) + o(s) \, ,\\
\bar{q}(t) = q(t) + s\xi(t,q) + o(s),
\end{cases}
\end{equation}
where $\eta\in \mathcal{C}^1(\mathbb{R})$ and
$\xi\in \mathcal{C}^1(\mathbb{R}^{2})$. We say that
functional \eqref{agp} is invariant under the $s$-parameter group of
infinitesimal transformations \eqref{eq:tinf} up to the gauge-term
$\Phi$ if
\begin{multline}
\label{eq:invndLIP} \int_{I} \dot{\Phi}[q]_{\tau}(t)dt =
\frac{d}{ds} \int_{\bar{t}(I)}
F\left(t+s\eta(t,q(t))+o(s),q(t)+s\xi(t,q(t))+o(s),
\frac{\dot{q}(t)+s\dot{\xi}(t,q(t))}{1+s\dot{\eta}(t,q(t))},\right.\\
\left. q(t-\tau)+s\xi(t-\tau,q(t-\tau))+o(s),\frac{\dot{q}(t-\tau)
+s\dot{\xi}(t-\tau,q(t-\tau))}{1+s\dot{\eta}(t-\tau,q(t-\tau))}\right)
(1+s\dot{\eta}(t,q(t))) dt\Biggr|_{s=0}
\end{multline}
for any  subinterval $I \subseteq [t_1,t_2]$ and for all
$q(\cdot)\in Lip\left([t_1-\tau,t_2]\right)\,.$
\end{definition}

\begin{definition}[Isoperimetric constant of motion/isoperimetric conservation law with time delay]
\label{def:leicond} We say that a quantity
$C(t,t+\tau,q(t),q(t-\tau),q(t+\tau),\dot{q}(t),\dot{q}(t-\tau),\dot{q}(t+\tau))$
is an \emph{isoperimetric constant of motion with time delay} $\tau$
if
\begin{equation}
\label{eq:conslaw:td1} \frac{d}{dt}
C(t,t+\tau,q(t),q(t-\tau),q(t+\tau),\dot{q}(t),\dot{q}(t-\tau),\dot{q}(t+\tau))=
0
\end{equation}
along all the extremals $q(\cdot)$ (\textrm{cf.}
Definition~\ref{def:scale:ext}). The equality \eqref{eq:conslaw:td1}
is then an  \emph{isoperimetric conservation law with time delay}.
\end{definition}

The next theorem establishes an extension of Noether's theorem to
isoperimetric problems of the calculus of variations with time
delay.

\begin{theorem}[Isoperimetric Noether's symmetry theorem with time delay for Lipschitz functions \cite{GFML2016}]
\label{theo:tnnd} If functional \eqref{Pe} is invariant up to $\Phi$
in the sense of Definition~\ref{def:invndLIP}, then the quantity
$C(t,t+\tau,q(t),q(t-\tau),q(t+\tau),\dot{q}(t),\dot{q}(t-\tau),\dot{q}(t+\tau))$
defined by
\begin{multline}
\label{eq:tnnd} -\Phi[q]_{\tau}(t)+\left(\partial_{3} F[q]_{\tau}(t)
+\partial_{5} F[q]_{\tau}(t+\tau)\right)\cdot\xi(t,q(t))\\
+\Bigl(F[q]_{\tau}-\dot{q}(t)\cdot(\partial_{3} F[q]_{\tau}(t)
+\partial_{5} F[q]_{\tau}(t+\tau))\Bigr)\eta(t,q(t))
\end{multline}
for $t_1\leq t\leq t_{2}-\tau$ and by
\begin{equation}
\label{eq:Noeth} -\Phi[q]_{\tau}(t)+\partial_{3}
F[q]_{\tau}(t)\cdot\xi(t,q(t))
+\Bigl(F[q]_{\tau}-\dot{q}(t)\cdot\partial_{3}
F[q]_{\tau}(t)\Bigr)\eta(t,q(t))
\end{equation}
for $t_2-\tau\leq t\leq t_{2}\,,$ is a constant of motion with time
delay along any $q(\cdot)\in
Lip\left([t_1-\tau,t_2]\right)$ satisfying both
\eqref{eq:eldf11} and \eqref{eq:cdrnd}-\eqref{eq:cdrnd1}, i.e.,
along any Lipschitz Euler--Lagrange extremal that is also a
Lipschitz DuBois--Reymond extremal that satisfy the condition
\eqref{CDUR}.
\end{theorem}


\section{Nonsmooth higher-order Noether's theorem for isoperimetric
problems of the calculus of variations with time delay}
\label{sec:MRHO}

Let $\mathbb{W}^{k,p}$, $k\geq 1$, $1\leq p \leq \infty$, denote the
class of functions that are absolutely continuous with their
derivatives up to order $k-1$, the $k$th derivative belonging to
$L^p$. With this notation, the class $Lip$ of Lipschitz functions is
represented by $\mathbb{W}^{1,\infty}$. We now extend previous
results to isoperimetric problems with higher-order derivatives.


\subsection{Higher-order Euler--Lagrange and DuBois--Reymond optimality conditions with time delay}
\label{HOEL}
Let $m\in\mathbb{N}$ and $q^{(i)}(t)$ denote the $i$th
derivative of the vector $q(t)$ defined in $\mathbb{R}^n$ ($n\in \mathbb{N}^*$), $i=0,\dots,m$, with $q^{(0)}(t)=q(t)$. For
simplicity of notation, we introduce the operator $[\cdot]^m_{\tau}$
by
$$
[q]^m_{\tau}(t) := \Bigl(t,q(t),\dot{q}(t),
\ldots,q^{(m)}(t),\\q(t-\tau),\dot{q}(t-\tau),
\ldots,q^{(m)}(t-\tau)\Bigr).
$$
Consider the following higher-order isoperimetric variational
problem with time delay:
\begin{problem}\label{Pb1m}
To minimize
\begin{equation}
\label{Pm} J^{\tau}_{m}[q(\cdot)] =\int_{t_1}^{t_2} L[q]^m_{\tau}(t)
dt
\end{equation}
subject to  the isoperimetric equality constraints
\begin{equation}
\label{C1T} I^{\tau}_{m}[q(\cdot)]=\int_{t_1}^{t_2} g[q]^m_{\tau}(t)
dt=l,\,\,\,l\in\mathbb{R}^k\,,
\end{equation}
boundary conditions \eqref{Pe2} and
$q^{(i)}(t_2)=q_{t_2}^i,~i=1,\dots,m-1$. The functions $L,g
:[t_1,t_2] \times \mathbb{R}^{2 n (m+1)} \rightarrow \mathbb{R}$ are
assumed to be a $\mathcal{C}^{m+1}$-function with respect to all
their arguments, admissible functions $q(\cdot)$ are assumed to be
$\mathbb{W}^{m,\infty}$, $t_{1}< t_{2}$ are fixed in $\mathbb{R}$,
$\tau$ is a given positive real number such that $\tau<t_{2}-t_{1}$,
$q_{t_2}^i$ are given vectors in $\mathbb{R}^n$, $i=1,\dots,m-1\,,$
 $l$ is a specified real constant and $\delta$ is a given
piecewise smooth function on $[t_1-\tau,t_1]$.
\end{problem}

\begin{remark}
When $m=1$ and $n=1$ the Problem~\ref{Pb1m} reduces to Problem~\ref{Pb1}.
\end{remark}

In \cite{GF2013} the authors proved the following corollary:
\begin{corollary}[Higher-order Euler--Lagrange equations with time delay in differential form] If
$q(\cdot)\in\mathbb{W}^{m,\infty}\left([t_1-\tau,t_2],
\mathbb{R}^{n}\right)$ is an extremal of functional \eqref{Pm},
 then
\begin{equation*}
\sum_{i=0}^{m}(-1)^{i}\frac{d^{i}}{dt^{i}}\Bigl(\partial_{i+2}
L[q]^m_{\tau}(t)+\partial_{i+m+3} L[q]^m_{\tau}(t+\tau)\Bigr)=0
\end{equation*}
for $t_1\leq t\leq t_{2}-\tau$ and
\begin{equation*}
\sum_{i=0}^{m}(-1)^{i}\frac{d^{i}}{dt^{i}}\partial_{i+2}
L[q]^m_{\tau}(t)=0
\end{equation*}
for $t_{2}-\tau\leq t \leq t_{2}$.
\end{corollary}

The previous corollary motivates the following definition.

\begin{definition} An admissible function
$q(\cdot)\in\mathbb{W}^{m,\infty}\left([t_1-\tau,t_2],
\mathbb{R}^{n}\right)$ is an extremal for problem
\eqref{C1T}--\eqref{Pe2} if it satisfies the following
Euler--Lagrange equations with time delay:
\begin{equation}\label{elg}
\sum_{i=0}^{m}(-1)^{i}\frac{d^{i}}{dt^{i}}\Bigl(\partial_{i+2}
g[q]^m_{\tau}(t)+\partial_{i+m+3} g[q]^m_{\tau}(t+\tau)\Bigr)=0
\end{equation}
for $t_1\leq t\leq t_{2}-\tau$ and
\begin{equation}\label{elg1}
\sum_{i=0}^{m}(-1)^{i}\frac{d^{i}}{dt^{i}}\partial_{i+2}
g[q]^m_{\tau}(t)=0
\end{equation}
for $t_{2}-\tau\leq t \leq t_{2}$.
\end{definition}

Now we extend notion of normal extremizer
(Definition~\ref{def:extr}) to higher-order normal extremizer for
isoperimetric problems of the calculus of variations with time
delay.

\begin{definition}\label{def:extr1}
An extremizer of Problem~\ref{Pb1m} that does not satisfy
\eqref{elg}--\eqref{elg1} is said to be a \emph{higher-order normal
extremizer}; otherwise (i.e., if it satisfies
\eqref{elg}--\eqref{elg1} for all $t\in[t_1,t_2]$), is said to be
\emph{higher-order abnormal extremizer}.
\end{definition}

The next theorem is crucial for our purposes.

\begin{theorem}[Isoperimetric higher-order Euler--Lagrange equations with time delay in integral form]
\label{Thm:ELdeordm} If
$q(\cdot)\in\mathbb{W}^{m,\infty}\left([t_1-\tau,t_2],
\mathbb{R}^{n}\right)$ is a higher-order normal extremizer of
Problem~\ref{Pb1m}, then $q(\cdot)$ satisfies the following
isoperimetric higher-order Euler--Lagrange integral equations  with
time delay:
\begin{multline}
\label{eq:ELdeordmInt}
\sum_{i=0}^{m}(-1)^{m-i-1}\Biggl(\underbrace{\int_{t_2-\tau}^{t}
\int_{t_2-\tau}^{s_1}\dots
\int_{t_2-\tau}^{s_{m-i-1}}}_{m-i~\textnormal{times}}
\Bigl(\partial_{i+2}
F[q]^m_{\tau}(s_{m-i})\\
+\partial_{i+m+3} F[q]^m_{\tau}(s_{m-i}+\tau)\Bigr)ds_{m-i}\dots
ds_2 ds_1\Biggr)=p(t)
\end{multline}
for $t_1\leq t\leq t_{2}-\tau$ and
\begin{equation}
\label{eq:ELdeordmInt1}
\sum_{i=0}^{m}(-1)^{m-i-1}\Biggl(\underbrace{\int_{t_2-\tau}^{t}
\int_{t_2-\tau}^{s_1}\dots
\int_{t_2-\tau}^{s_{m-i-1}}}_{m-i~\textnormal{times}}\Bigl(\partial_{i+2}
F[q]^m_{\tau}(t)\Bigr)ds_{m-i}\dots ds_2 ds_1\Biggr) =p(t)
\end{equation}
for $t_{2}-\tau\leq t \leq t_{2}$, where $F$ is the augmented
Lagrangian \eqref{eq:aug:Lag} associated with Problem~\ref{Pb1m},
$p(t)$ is a polynomial of order $m-1$, i.e., $p(t)=c_0+c_1t+\dots
+c_{m-1}t^{m-1}$ for some constants $c_i\in\mathbb{R}$,
$i=0,\dots,m-1$.
\end{theorem}

\begin{proof}Consider neighboring functions of the form
\begin{equation}
\label{admfunct} \hat{q}(t)=q(t)+\epsilon_1h_1(t)+\epsilon_2h_2(t),
\end{equation}
where for each $\kappa\in\{1,2\}$ $\epsilon_\kappa$ is a
sufficiently small parameter, $h_\kappa\in
\mathbb{W}^{m,\infty}\left([t_1-\tau,t_2], \mathbb{R}^{n}\right)$,
$h^{(i)}_{\kappa}(t_2)=h^{(i)}_{\kappa}(t_2-\tau)=0,\,i=0,\ldots,m-1$,
and $h_\kappa(t)=0$, $t \in[t_{1}-\tau,t_{1}]\,.$

First we will show that (\ref{admfunct}) has a subset of admissible
functions for the variational isoperimetric problem with time delay.
Consider the quantity
$$
I^{\tau}_m[\hat{q}(\cdot)]\\=\int_{t_1}^{t_2}
g[\hat{q}]^m_{\tau}(t)dt.
$$
Then we can regard $I^{\tau}_{m}[\hat{q}(\cdot)]$ as a function of
$\epsilon_1$ and $\epsilon_2$. Define
$\hat{I}(\epsilon_1,\epsilon_2)=I^{\tau}_{m}[\hat{q}(\cdot)]-l$.
Thus,
\begin{equation}
\label{implicit1} \hat I(0,0)=0\,.
\end{equation}

On the other hand, we have
\begin{equation}
\label{pel} \left.\frac{\partial \hat I}{\partial \epsilon_2}
\right|_{(0,0)}=\int_{t_1}^{t_2}\left(\sum_{i=0}^{m}\partial_{i+2}
g[q]^m_{\tau}(t)\cdot h^{(i)}_{2}(t)+\sum_{i=0}^{m}\partial_{i+m+3}
g[q]^m_{\tau}(t)\cdot h^{(i)}_{2}(t-\tau)\right)dt\,.
\end{equation}
Performing the linear change of variables $t=\sigma+\tau$ in the
last term of integral \eqref{pel}, and using the fact that
$h_{2}(t)=0$ if $t \in[t_{1}-\tau,t_{1}]$, \eqref{pel} becomes
\begin{multline}
\label{pel1} \left.\frac{\partial \hat I}{\partial \epsilon_2}
\right|_{(0,0)}=\int_{t_1}^{t_2}\left(\sum_{i=0}^{m}\partial_{i+2}
g[q]^m_{\tau}(t)\cdot h^{(i)}_{2}(t)\right)dt\\
+\int_{t_1}^{t_2-\tau}\left(\sum_{i=0}^{m}\partial_{i+m+3}
g[q]^m_{\tau}(t+\tau)\cdot h^{(i)}_{2}(t)\right)dt\,.
\end{multline}
By repeated integration by parts one has
\begin{multline}
\label{eq:identity1}
\sum\limits_{i=0}^{m} \int_{t_1}^{t_2} \partial_{i+2}g[q]^m_{\tau}(t)\cdot h^{(i)}_{2}(t) dt\\
=\sum\limits_{i=0}^{m}\Biggl\{\Biggl[\sum\limits_{j=1}^{m-i}(-1)^{j+1}h^{(i+j-1)}_{2}(t)
\cdot\Biggl(\underbrace{\int_{t_2-\tau}^{t}\int_{t_2-\tau}^{s_1}
\dots
\int_{t_2-\tau}^{s_{j-1}}}_{j~\textnormal{times}}\Bigl(\partial_{i+2}
g[q]^m_{\tau}(s_j)\Bigr)ds_{j}\dots ds_2 ds_1\Biggr)\Biggr]_{t_1}^{t_2}\\
+(-1)^{i}\int_{t_1}^{t_2}h^{(m)}_{2}(t) \cdot\Biggl(
\underbrace{\int_{t_2-\tau}^{t}\int_{t_2-\tau}^{s_1}\dots
\int_{t_2-\tau}^{s_{m-i-1}}}_{m-i~\textnormal{times}}\Bigl(\partial_{i+2}
g[q]^m_{\tau}(s_{m-i})\Bigr)ds_{m-i} \dots
ds_{2}ds_1\Biggr)dt\Biggr\}
\end{multline}
and
\begin{multline}
\label{eq:identity2} \sum\limits_{i=0}^{m}\int_{t_1}^{t_2-\tau}
\partial_{i+m+3}g[q]^m_{\tau}(t+\tau)\cdot h^{(i)}_{2}(t) dt\\
=\sum\limits_{i=0}^{m}\Biggl\{\Biggl[\sum\limits_{j=1}^{m-i}(-1)^{j+1}h^{(i+j-1)}_{2}(t)
\cdot\Biggl(\underbrace{\int_{t_2-\tau}^{t}\int_{t_2-\tau}^{s_1}
\dots
\int_{t_2-\tau}^{s_{j-1}}}_{j~\textnormal{times}}\Bigl(\partial_{i+m+3}
g[q]^m_{\tau}(s_j+\tau)\Bigr)ds_{j}\dots ds_2 ds_1\Biggr)\Biggr]_{t_1}^{t_2-\tau}\\
+(-1)^{i}\int_{t_1}^{t_2-\tau}h^{(m)}_{2}(t) \cdot\Biggl(
\underbrace{\int_{t_2-\tau}^{t}\int_{t_2-\tau}^{s_1} \dots
\int_{t_2-\tau}^{s_{m-i-1}}}_{m-i~\textnormal{times}}\Bigl(\partial_{i+m+3}
g[q]^m_{\tau}(s_{m-i}+\tau)\Bigr)ds_{m-i} \dots
ds_{2}ds_1\Biggr)dt\Biggr\}.
\end{multline}
Because $h^{(i)}_{2}(t_2)=0,\,i=0,\ldots,m-1$, and $h_{2}(t)=0$, $t
\in[t_{1}-\tau,t_{1}]$, the terms without integral sign in the
right-hand sides of identities \eqref{eq:identity1} and
\eqref{eq:identity2} vanish. Therefore, equation \eqref{pel1}
becomes
\begin{multline}
\label{eq:identity3}  \left.\frac{\partial \hat I}{\partial
\epsilon_2} \right|_{(0,0)}=\int_{t_1}^{t_2-\tau}h^{(m)}_{2}(t)
\cdot\Biggl[\sum\limits_{i=0}^{m}(-1)^{i} \Biggl(
\underbrace{\int_{t_2-\tau}^{t}\int_{t_2-\tau}^{s_1} \dots
\int_{t_2-\tau}^{s_{m-i-1}}}_{m-i~\textnormal{times}}\Bigl(\partial_{i+2}
g[q]^m_{\tau}(s_{m-i})\\
+\partial_{i+m+3}
g[q]^m_{\tau}(s_{m-i}+\tau)\Bigr)ds_{m-i} \dots ds_{2}ds_1\Biggr)\Biggr]dt\\
+\int_{t_2-\tau}^{t_2}h^{(m)}_{2}(t)
\cdot\Biggl[\sum\limits_{i=0}^{m}(-1)^{i} \Biggl(
\underbrace{\int_{t_2-\tau}^{t}\int_{t_2-\tau}^{s_1} \dots
\int_{t_2-\tau}^{s_{m-i-1}}}_{m-i~\textnormal{times}}\Bigl(\partial_{i+2}
g[q]^m_{\tau}(s_{m-i})\Bigr)ds_{m-i} \dots
ds_{2}ds_1\Biggr)\Biggr]dt\,.
\end{multline}
For $i=0,\dots,m$ we define functions
\begin{equation*}
\varphi_i (t)=
\begin{cases}
\partial_{i+2}g[q]^m_{\tau}(t)+\partial_{i+m+3}g[q]^m_{\tau}(t+\tau)
& ~\textnormal{for}~ t_1\leq t\leq t_2-\tau\\
\partial_{i+2}g[q]^m_{\tau}(t) & ~\textnormal{for}~ t_2-\tau\leq t\leq t_2.
\end{cases}
\end{equation*}
Then one can write equation \eqref{eq:identity3} as follows:
\begin{equation*}
\left.\frac{\partial \hat I}{\partial \epsilon_2}
\right|_{(0,0)}=\int_{t_1}^{t_2}h^{(m)}_{2}(t)
\cdot\Biggl[\sum\limits_{i=0}^{m}(-1)^{i} \Biggl(
\underbrace{\int_{t_2-\tau}^{t}\int_{t_2-\tau}^{s_1} \dots
\int_{t_2-\tau}^{s_{m-i-1}}}_{m-i~\textnormal{times}}
\Bigl(\varphi_i (s_{m-i})\Bigr)ds_{m-i} \dots
ds_{2}ds_1\Biggr)\Biggr]dt\,.
\end{equation*}

 Since $q(\cdot)\in\mathbb{W}^{m,\infty}\left([t_1-\tau,t_2],
\mathbb{R}^{n}\right)$ is a higher-order normal extremizer of
Problem~\ref{Pb1m}, by the fundamental lemma of the calculus of
variations (see, \textrm{e.g.}, \cite{Bruce:book}), there exists a
function $h_2$ such that
\begin{equation}
\label{implicit2} \left.\frac{\partial \hat I}{\partial \epsilon_2}
\right|_{(0,0)}\neq 0\,.
\end{equation}
Using (\ref{implicit1}) and (\ref{implicit2}), the implicit function
theorem asserts that there exists a function $\epsilon_2(\cdot)$,
defined in a neighborhood of zero, such that $\hat
I(\epsilon_1,\epsilon_2(\epsilon_1))=0$. Consider the real function
$\hat J(\epsilon_1,\epsilon_2)=J^{\tau}_{m}[\hat{q}(\cdot)]$. By
hypothesis, $\hat J$ has minimum (or maximum) at $(0,0)$ subject to
the constraint $\hat I(0,0)=0$, and we have proved that $\nabla \hat
I(0,0)\neq \textbf{0}$. We can appeal to the Lagrange multiplier
rule (see, \textrm{e.g.}, \cite[p.~77]{Bruce:book}) to assert the
existence of a number $\lambda$ such that $\nabla(\hat
J(0,0)-\lambda \cdot\hat I(0,0))=\textbf{0}$.

Repeating the calculations as before,

\begin{equation*}
\left.\frac{\partial \hat J}{\partial \epsilon_1}
\right|_{(0,0)}=\int_{t_1}^{t_2}h^{(m)}_{1}(t)
\cdot\Biggl[\sum\limits_{i=0}^{m}(-1)^{i} \Biggl(
\underbrace{\int_{t_2-\tau}^{t}\int_{t_2-\tau}^{s_1} \dots
\int_{t_2-\tau}^{s_{m-i-1}}}_{m-i~\textnormal{times}} \Bigl(\phi_i
(s_{m-i})\Bigr)ds_{m-i} \dots ds_{2}ds_1\Biggr)\Biggr]dt\,.
\end{equation*}
and

\begin{equation*}
\left.\frac{\partial \hat I}{\partial \epsilon_1}
\right|_{(0,0)}=\int_{t_1}^{t_2}h^{(m)}_{1}(t)
\cdot\Biggl[\sum\limits_{i=0}^{m}(-1)^{i} \Biggl(
\underbrace{\int_{t_2-\tau}^{t}\int_{t_2-\tau}^{s_1} \dots
\int_{t_2-\tau}^{s_{m-i-1}}}_{m-i~\textnormal{times}}
\Bigl(\varphi_i (s_{m-i})\Bigr)ds_{m-i} \dots
ds_{2}ds_1\Biggr)\Biggr]dt
\end{equation*}
where for $i=0,\dots,m$
\begin{equation*}
\phi_i (t)=
\begin{cases}
\partial_{i+2}L[q]^m_{\tau}(t)+\partial_{i+m+3}L[q]^m_{\tau}(t+\tau)
& ~\textnormal{for}~ t_1\leq t\leq t_2-\tau\\
\partial_{i+2}L[q]^m_{\tau}(t) & ~\textnormal{for}~ t_2-\tau\leq t\leq t_2.
\end{cases}
\end{equation*}

 Therefore, one has

 \begin{multline}\label{c123}
\int_{t_1}^{t_2}h^{(m)}_{1}(t)
\cdot\Biggl[\sum\limits_{i=0}^{m}(-1)^{i} \Biggl(
\underbrace{\int_{t_2-\tau}^{t}\int_{t_2-\tau}^{s_1} \dots
\int_{t_2-\tau}^{s_{m-i-1}}}_{m-i~\textnormal{times}} \Bigl(\phi_i
(s_{m-i})\Bigr)ds_{m-i} \dots ds_{2}ds_1\Biggr)\\
-\lambda\cdot \sum\limits_{i=0}^{m}(-1)^{i} \Biggl(
\underbrace{\int_{t_2-\tau}^{t}\int_{t_2-\tau}^{s_1} \dots
\int_{t_2-\tau}^{s_{m-i-1}}}_{m-i~\textnormal{times}}
\Bigl(\varphi_i (s_{m-i})\Bigr)ds_{m-i} \dots
ds_{2}ds_1\Biggr)\Biggr]dt=0\,.
\end{multline}

Applying the higher-order DuBois--Reymond lemma in \eqref{c123}
(see, \textrm{e.g.}, \cite{Jost:book,Troutman:book}), one arrives to
\eqref{eq:ELdeordmInt} and \eqref{eq:ELdeordmInt1}.
\end{proof}

\begin{corollary}[Isoperimetric higher-order Euler--Lagrange equations with time delay in differential form]
\label{cor:16} If
$q(\cdot)\in\mathbb{W}^{m,\infty}\left([t_1-\tau,t_2],
\mathbb{R}^{n}\right)$ is a higher-order normal extremizer of
Problem~\ref{Pb1m}, then
\begin{equation}
\label{eq:ELdeordm}
\sum_{i=0}^{m}(-1)^{i}\frac{d^{i}}{dt^{i}}\Bigl(\partial_{i+2}
F[q]^m_{\tau}(t)+\partial_{i+m+3} F[q]^m_{\tau}(t+\tau)\Bigr)=0
\end{equation}
for $t_1\leq t\leq t_{2}-\tau$ and
\begin{equation}
\label{eq:ELdeordm1}
\sum_{i=0}^{m}(-1)^{i}\frac{d^{i}}{dt^{i}}\partial_{i+2}
F[q]^m_{\tau}(t)=0
\end{equation}
for $t_{2}-\tau\leq t \leq t_{2}\,.$
\end{corollary}

\begin{proof}
We obtain \eqref{eq:ELdeordm} and \eqref{eq:ELdeordm1} applying the
derivative of order $m$ to \eqref{eq:ELdeordmInt} and
\eqref{eq:ELdeordmInt1}, respectively.
\end{proof}

\begin{remark}
If $m=1$ and $n=1$, then the higher-order Euler--Lagrange equations
\eqref{eq:ELdeordm}--\eqref{eq:ELdeordm1} reduce to
\eqref{eq:eldf11}.
\end{remark}

Associated to a given function
$q(\cdot)\in\mathbb{W}^{m,\infty}\left([t_1-\tau,t_2],
\mathbb{R}^{n}\right)$, it is convenient to introduce the following
quantities (\textrm{cf.} \cite{Torres:proper}):
\begin{equation}
\label{eq:eqprin}
\psi^{j}_1=\sum_{i=0}^{m-j}(-1)^{i}\frac{d^{i}}{dt^{i}}\Bigl(\partial_{i+j+2}
F[q]^m_{\tau}(t)+\partial_{i+j+m+3} F[q]^m_{\tau}(t+\tau)\Bigr)
\end{equation}
for $t_1\leq t\leq t_{2}-\tau$, and
\begin{equation}
\label{eq:eqprin11}
\psi^{j}_2=\sum_{i=0}^{m-j}(-1)^{i}\frac{d^{i}}{dt^{i}}\partial_{i+j+2}
F[q]^m_{\tau}(t)
\end{equation}
for $t_{2}-\tau\leq t\leq t_{2}$, where $j=0,\ldots,m$. These
operators are useful for our purposes because of the following
properties:
\begin{equation}
\label{eq:eqprin1}
\frac{d}{dt}\psi^{j}_1=\partial_{j+1}F[q]^m_{\tau}(t)
+\partial_{j+m+2}F[q]^m_{\tau}(t+\tau)-\psi^{j-1}_1
\end{equation}
for $t_1\leq t\leq t_{2}-\tau$, and
\begin{equation*}
\frac{d}{dt}\psi^{j}_2=\partial_{j+1}F[q]^m_{\tau}(t) -\psi^{j-1}_2
\end{equation*}
for $t_{2}-\tau\leq t\leq t_{2}$, where $j=1,\ldots,m$. We are now
in conditions to prove a higher-order DuBois--Reymond optimality
condition for isoperimetric problems with time delay.

\begin{theorem}[Isoprimetric higher-order delayed DuBois--Reymond condition]
\label{theo:cDRifm} If\\
$q(\cdot)\in\mathbb{W}^{m,\infty}\left([t_1-\tau,t_2],
\mathbb{R}^{n}\right)$ is a higher-order normal extremizer of
Problem~\ref{Pb1m} such that
\begin{equation}\label{CDUR1}
\sum_{j=0}^{m}
\partial_{j+m+3} F[q]^m_{\tau}(t+\tau)\cdot
q^{(j+1)}(t)=0
\end{equation}

for all $t\in[t_1-\tau,t_2-\tau]$, then it satisfies the following
conditions:
\begin{equation}
\label{eq:DBRordm} \frac{d}{dt}\left(F[q]^m_{\tau}(t)
-\sum_{j=1}^{m}\psi^{j}_1\cdot q^{(j)}(t)\right) =\partial_{1}
F[q]^m_{\tau}(t)
\end{equation}
for $t_1\leq t\leq t_{2}-\tau$ and
\begin{equation}
\label{eq:DBRordm:2} \frac{d}{dt}\left(F[q]^m_{\tau}(t)
-\sum_{j=1}^{m}\psi^{j}_2\cdot q^{(j)}(t)\right) =\partial_{1}
F[q]^m_{\tau}(t)
\end{equation}
for $t_{2}-\tau\leq t\leq t_{2}$, where $F$ is the augmented
Lagrangian \eqref{eq:aug:Lag} associated with Problem~\ref{Pb1m},
$\psi^{j}_1$ given by \eqref{eq:eqprin} and $\psi^{j}_2$ by
\eqref{eq:eqprin11}.
\end{theorem}

\begin{proof}
We prove the theorem in the interval $t_{1}\leq t\leq t_{2}-\tau$.
The proof is similar for $t_{2}-\tau\leq t\leq t_{2}$. We derive
equation \eqref{eq:DBRordm} as follows:

Let an arbitrary  $x\in [t_1, t_2-\tau]\,.$ Note that
\begin{multline}
\label{pr} \int_{t_1}^{x}\frac{d}{dt}\left(F[q]^m_{\tau}(t)
-\sum_{j=1}^{m}\psi^{j}_1\cdot
q^{(j)}(t)\right)dt=\int_{t_1}^{x}\frac{d}{dt}\left(L[q]^m_{\tau}(t)-\lambda\cdot
g[q]^m_{\tau}(t)
-\sum_{j=1}^{m}\psi^{j}_1\cdot q^{(j)}(t)\right)dt\\
=\int_{t_1}^{x}\left(\partial_1 \left(L[q]^m_{\tau}(t)-\lambda\cdot
g[q]^m_{\tau}(t)\right)+\sum_{j=0}^{m}
\partial_{j+2} \left(L[q]^m_{\tau}(t)-\lambda\cdot
g[q]^m_{\tau}(t)\right)\cdot q^{(j+1)}(t)\right.\\
\left.-\sum_{j=1}^{m}\left(\dot{\psi}^{j}_1\cdot
q^{(j)}(t)+\psi^{j}_1\cdot q^{(j+1)}(t)\right)\right)dt\\
+\int_{t_1}^{x}\sum_{j=0}^{m}
\partial_{j+m+3} \left(L[q]^m_{\tau}(t)-\lambda\cdot
g[q]^m_{\tau}(t)\right)\cdot q^{(j+1)}(t-\tau)dt\,.
\end{multline}
Observe that, by hypothesis \eqref{CDUR1}, the last integral of
\eqref{pr} is null and using \eqref{eq:eqprin1}, the equation
\eqref{pr} becomes
\begin{multline}
\label{pr2}
\int_{t_1}^{x}\frac{d}{dt}\left(L[q]^m_{\tau}(t)-\lambda\cdot
g[q]^m_{\tau}(t) -\sum_{j=1}^{m}\psi^{j}_1\cdot
q^{(j)}(t)\right)dt\\
=\int_{t_1}^{x}\left[\partial_1 \left(L[q]^m_{\tau}(t)-\lambda\cdot
g[q]^m_{\tau}(t)\right)+\sum_{j=0}^{m}
\partial_{j+2} \left(L[q]^m_{\tau}(t)-\lambda\cdot
g[q]^m_{\tau}(t)\right)\cdot q^{(j+1)}(t)\right.\\ \left.
-\sum_{j=1}^{m}\left(\left(\partial_{j+1}L[q]^m_{\tau}(t)
+\partial_{j+m+2}L[q]^m_{\tau}(t+\tau)-\psi^{j-1}_1\right)\cdot
q^{(j)}(t)+\psi^{j}_1\cdot q^{(j+1)}(t)\right)\right]dt.
\end{multline}
We now simplify the second term on the right-hand side of
\eqref{pr2}:
\begin{multline}
\label{pr3} \sum_{j=1}^{m}\left(\left(\partial_{j+1}L[q]^m_{\tau}(t)
+\partial_{j+m+2}L[q]^m_{\tau}(t+\tau)-\psi^{j-1}_1\right)\cdot
q^{(j)}(t)+\psi^{j}_1\cdot q^{(j+1)}(t)\right)\\
=\sum_{j=0}^{m-1}\Bigl(\left(\partial_{j+2}L[q]^m_{\tau}(t)
+\partial_{j+m+3}L[q]^m_{\tau}(t+\tau)-\psi^{j}_1\right)
\cdot q^{(j+1)}(t)+\psi^{j+1}_1\cdot q^{(j+2)}(t)\Bigr)\\
=\sum_{j=0}^{m-1}\left[\left(\partial_{j+2}L[q]^m_{\tau}(t)
+\partial_{j+m+3}L[q]^m_{\tau}(t+\tau)\right)\cdot
q^{(j+1)}(t)\right] -\psi^{0}_1\cdot \dot{q}(t)+\psi^{m}_1\cdot
q^{(m+1)}(t).
\end{multline}
Substituting \eqref{pr3} into \eqref{pr2} and using the higher-order
Euler--Lagrange equations with time delay \eqref{eq:ELdeordm}, and
since, by
definition,
\begin{equation*}
\psi^{m}_1=\partial_{m+2}L[q]^m_{\tau}(t)+\partial_{2m+3}L[q]^m_{\tau}(t+\tau)\,,
\end{equation*}

\begin{equation*}
\psi^0_1=
\sum_{i=0}^{m}(-1)^{i}\frac{d^{i}}{dt^{i}}\Bigl(\partial_{i+2}
L[q]^m_{\tau}(t)+\partial_{i+m+3} L[q]^m_{\tau}(t+\tau)\Bigr)=0
\end{equation*}
and by hypothesis \eqref{CDUR1}
\begin{equation*}
\sum_{j=0}^{m-1} \partial_{j+m+3}L[q]^m_{\tau}(t+\tau)\cdot
q^{(j+1)}(t)=-\partial_{2m+3}L[q]^m_{\tau}(t+\tau)\cdot
q^{(m+1)}(t)\,,
\end{equation*}
we conclude that
\begin{multline*}
\int_{t_1}^{x}\frac{d}{dt}\left(L[q]^m_{\tau}(t)
-\sum_{j=1}^{m}\psi^{j}_1\cdot q^{(j)}(t)\right)dt\\
=\int_{t_1}^{x}\left[\partial_1 L[q]_{\tau}^m(t)
+\left(\partial_{m+2}L[q]^m_{\tau}(t)+\partial_{2m+3}
L[q]^m_{\tau}(t+\tau)\right)\cdot q^{(m+1)}\right.\\
\left. +\psi^{0}_1\cdot \dot{q}(t)-\psi^{m}_1\cdot
q^{(m+1)}(t)\right]dt =\int_{t_1}^{x}\partial_1
L[q]_{\tau}^m(t)dt\,.
\end{multline*}

We finally obtain \eqref{eq:DBRordm} by the arbitrariness
$x\in[t_{1},t_{2}-\tau]\,.$
\end{proof}

In the particular case when $m=1$ and $n=1$, we obtain from
Theorem~\ref{theo:cDRifm} an extension of Theorem~\ref{theo:cdrnd}
to the class of Lipschitz functions.

\begin{corollary}[Nonsmooth isoperimetric DuBois--Reymond conditions]
\label{cor:DR:m1}  If \\ $q(\cdot)\in
Lip\left([t_1-\tau,t_2];\mathbb{R}\right)$ is a normal
isoperimetric extremals with time delay, then the DuBois--Reymond
conditions with time delay \eqref{eq:cdrnd} and \eqref{eq:cdrnd1}
hold true.
\end{corollary}

\begin{proof}
For $m=1$, the hypothesis \eqref{CDUR1} is reduced to \eqref{CDUR},
condition \eqref{eq:DBRordm} to
\begin{equation}
\label{eq:DBRordm111} \frac{d}{dt}\left(F[q]_{\tau}(t)
-\psi^{1}_1\cdot \dot{q}(t)\right) =\partial_{1} F[q]_{\tau}(t)
\end{equation}
for $t_1\leq t\leq t_{2}-\tau$, and \eqref{eq:DBRordm:2} to
\begin{equation}
\label{eq:DBRordm121} \frac{d}{dt}\left(F[q]_{\tau}(t)
-\psi^{1}_2\cdot \dot{q}(t)\right) =\partial_{1} F[q]_{\tau}(t)
\end{equation}
for $t_{2}-\tau\leq t\leq t_{2}$. Keeping in mind  \eqref{eq:eqprin}
and \eqref{eq:eqprin11}, we obtain
\begin{equation}
\label{eq:DBRord3} \psi^{1}_1=\partial_3 F[q]_{\tau}(t)+\partial_5
F[q]_{\tau}(t+\tau)
\end{equation}
and
\begin{equation}
\label{eq:DBRord4} \psi^{1}_2=\partial_3 F[q]_{\tau}(t).
\end{equation}
One finds the intended equalities \eqref{eq:cdrnd} and
\eqref{eq:cdrnd1} by substituting the quantities \eqref{eq:DBRord3}
and \eqref{eq:DBRord4} into \eqref{eq:DBRordm111} and
\eqref{eq:DBRordm121}, respectively.
\end{proof}


\subsection{Isoperimetric higher-order Noether's symmetry theorem with time delay}
\noindent
 Now, we generalize the Noether-type theorem
(Theorem~\ref{theo:tnnd}) to the more general case of delayed
isoperimetric variational problems with higher-order derivatives.

We know (see Section~\ref{HOEL}) that by using the Lagrange
multiplier rule, Problem~\ref{Pb1m} is equivalent to the following
augmented problem: to minimize
\begin{equation}
\label{agp1}
\begin{split}
J^{\tau}_{m}[q(\cdot),\lambda]
&= \int_{t_1}^{t_2} F[q,\lambda]_{\tau}^{m}(t)dt\\
&:=\int_{t_1}^{t_2} \left[L[q]_{\tau}^{m}(t) -\lambda \cdot
g[q]_{\tau}^{m}(t)\right] dt
\end{split}
\end{equation}
subject to boundary conditions \eqref{Pe2} and
$q^{(i)}(t_2)=q_{t_2}^i,~i=1,\dots,m-1$.

The notion of variational invariance for Problem~\ref{Pb1m} is
defined with the help of the equivalent augmented Lagrangian
\eqref{agp1}.

\begin{definition}[Invariance of \eqref{agp1} up to a gauge-term]
\label{def:invaifm} Consider the $s$-parameter group of
infinitesimal transformations \eqref{eq:tinf}. Functional
\eqref{agp1} is invariant under \eqref{eq:tinf} up to the gauge-term
$\Phi$ if
\begin{multline}
\label{eq:invndm} \int_{I} \dot{\Phi}[q]^m_{\tau}(t)dt =
\frac{d}{ds} \int_{\bar{t}(I)} F\left(\bar{t},\bar{q}(\bar{t}),
{\bar{q}}'(\bar{t}),\ldots,\bar{q}^{(m)}(\bar{t}),\right.\\
\left.\left.\bar{q}(\bar{t}-\tau),{\bar{q}}'(\bar{t}-\tau),
\ldots,\bar{q}^{(m)}(\bar{t}-\tau)\right)
(1+s\dot{\eta}(t,q(t))dt\right|_{s=0}
\end{multline}
for any  subinterval $I \subseteq [t_1,t_2]$ and for all
$q(\cdot)\in \mathbb{W}^{m,\infty}\left([t_1-\tau,t_2],
\mathbb{R}^{n}\right)$.
\end{definition}

\begin{remark}
Expressions $\dot{\Phi}$ and $\bar{q}^{(i)}$ in equation
\eqref{eq:invndm}, $i=1,\ldots,m$, are interpreted as
\begin{equation}
\label{eq:invifm1} \dot{\Phi}=\frac{d}{dt}\Phi\,\,,\quad\bar{q}'
=\frac{d\bar{q}}{d\bar{t}}=\frac{\frac{d\bar{q}}{dt}}{\frac{d\bar{t}}{dt}}\,\,,
\quad \bar{q}^{(i)}=\frac{d^{i}\bar{q}}{d\bar{t}^{i}}=
\frac{\frac{d}{dt}\left(\frac{d^{i-1}}{d\bar{t}^{i-1}}\bar{q}\right)}{\frac{d\bar{t}}{dt}},\,
i=2,\ldots,m.
\end{equation}
\end{remark}

The next lemma gives a necessary condition of invariance for
functional \eqref{agp1}.

\begin{lemma}[Necessary condition of invariance for \eqref{agp1}]
\label{thm:cnsi} If functional \eqref{agp1} is invariant up to the
gauge-term $\Phi$ under the $s$-parameter group of infinitesimal
transformations \eqref{eq:tinf}, then
\begin{multline}
\label{eq:cnsiifm}
\int_{t_1}^{t_2-\tau}\left[-\dot{\Phi}[q]^m_{\tau}(t)+\partial_{1}
F[q]^m_{\tau}(t)\eta(t,q)+ F[q]^m_{\tau}(t) \dot{\eta}(t,q)\right.\\
\left. +\sum_{i=0}^{m}\left(\partial_{i+2}
F[q]^m_{\tau}(t)+\partial_{i+m+3} F[q]^m_{\tau}(t+\tau)\right)\cdot
\rho^{i}(t) \right]dt =0
\end{multline}
for $t_1\leq t\leq t_{2}-\tau$ and
\begin{equation}
\label{eq:cnsiifm11}
\int_{t_2-\tau}^{t_2}\left[-\dot{\Phi}[q]^m_{\tau}(t)+\partial_{1}
F[q]^m_{\tau}(t)\eta(t,q)+ F[q]^m_{\tau}(t) \dot{\eta}(t,q) +
\sum_{i=0}^{m}\partial_{i+2} F[q]^m_{\tau}(t)\cdot
\rho^{i}(t)\right]dt =0
\end{equation}
for $t_2-\tau\leq t\leq t_{2}$, where
\begin{equation}
\label{eq:cnsiifm1}
\begin{cases}
\rho^{0}(t)=\xi(t,q) \, , \\
\rho^{i}(t)=\frac{d}{dt}\left(\rho^{i-1}(t)\right)-q^{(i)}(t)\dot{\eta}(t,q)\,
, \quad i=1,\ldots,m.
\end{cases}
\end{equation}
\end{lemma}

\begin{proof}
Without loss of generality, we take $I=[t_1,t_2]$. Then,
\eqref{eq:invndm} is equivalent to
\begin{multline}\label{mm}
\int_{t_1}^{t_2}\left[-\dot{\Phi}[q]_{\tau}(t)+\partial_{1}\left(L[q]_{\tau}^{m}(t)
-\lambda \cdot g[q]_{\tau}^{m}(t)\right)\eta(t,q)\right.\\
\left. +\sum_{i=0}^{m}\partial_{i+2} \left(L[q]_{\tau}^{m}(t)
-\lambda \cdot
g[q]_{\tau}^{m}(t)\right)\cdot\frac{\partial}{\partial s}
\left.\left(\frac{d^{i}\bar{q}}{d\bar{t}^{i}}\right)\right|_{s=0}\right.\\
\left. +\sum_{i=0}^{m}\partial_{i+m+3} \left(L[q]_{\tau}^{m}(t)
-\lambda \cdot
g[q]_{\tau}^{m}(t)\right)\cdot\frac{\partial}{\partial s}
\left.\left(\frac{d^{i}\bar{q}(\bar{t}-\tau)}{d(\bar{t}-\tau)^{i}}\right)\right|_{s=0}
\right.\\\left.+\left(L[q]_{\tau}^{m}(t) -\lambda \cdot
g[q]_{\tau}^{m}(t)\right) \dot{\eta} \right] =0.
\end{multline}
Using the fact that \eqref{eq:invifm1} implies
\begin{equation*}
\frac{\partial}{\partial
s}\left.\left(\frac{d\bar{q}(\bar{t})}{d\bar{t}}\right)\right|_{s=0}
=\dot{\xi}(t,q)-\dot{q}\dot{\eta}(t,q) \, ,
\end{equation*}
\begin{equation*}
\frac{\partial}{\partial s}\left.\left(
\frac{d^{i}\bar{q}(\bar{t})}{d\bar{t}^{i}}\right)\right|_{s=0}
=\frac{d}{dt}\left[\frac{\partial}{\partial s}
\left.\left(\frac{d^{i-1}\bar{q}(\bar{t})}{d\bar{t}^{i-1}}\right)\right|_{s=0}\right]
-q^{(i)}(t)\dot{\eta}(t,q)\, , \quad i=2,\ldots,m,
\end{equation*}
then equation \eqref{mm} becomes
\begin{multline}
\label{mm1}
\int_{t_1}^{t_2}\left[-\dot{\Phi}[q]^m_{\tau}(t)+\partial_{1}
\left(L[q]_{\tau}^{m}(t) -\lambda \cdot
g[q]_{\tau}^{m}(t)\right)\eta(t,q)+ \left(L[q]_{\tau}^{m}(t)
-\lambda \cdot g[q]_{\tau}^{m}(t)\right) \dot{\eta}(t,q)\right.\\
\left. +\sum_{i=0}^{m}\partial_{i+2} \left(L[q]_{\tau}^{m}(t)
-\lambda \cdot
g[q]_{\tau}^{m}(t)\right)\cdot\rho^{i}(t)\right.\\\left.+\sum_{i=0}^{m}\partial_{i+m+3}
\left(L[q]_{\tau}^{m}(t) -\lambda \cdot
g[q]_{\tau}^{m}(t)\right)\cdot \rho^{i}(t-\tau) \right]dt =0.
\end{multline}
Performing the linear change of variables $t=\sigma+\tau$ in the
last integral of \eqref{mm1}, and keeping in mind that $\xi=\eta= 0$
on $[t_1-\tau,t_1]$, equation \eqref{mm1} becomes
\begin{multline}
\label{mm2}
\int_{t_1}^{t_2-\tau}\left[-\dot{\Phi}[q]^m_{\tau}(t)+\partial_{1}
\left(L[q]_{\tau}^{m}(t) -\lambda \cdot
g[q]_{\tau}^{m}(t)\right)\eta(t,q)+ \left(L[q]_{\tau}^{m}(t)
-\lambda \cdot g[q]_{\tau}^{m}(t)\right) \dot{\eta}(t,q)\right.\\
\left. +\sum_{i=0}^{m}\left(\partial_{i+2} \left(L[q]_{\tau}^{m}(t)
-\lambda \cdot g[q]_{\tau}^{m}(t)\right)+\partial_{i+m+3}
\left(L[q]_{\tau}^{m}(t+\tau) -\lambda \cdot
g[q]_{\tau}^{m}(t+\tau)\right)\right)\cdot
\rho^{i}(t) \right]dt \\
+ \int_{t_2-\tau}^{t_2}\left[-\dot{\Phi}[q]^m_{\tau}(t)+\partial_{1}
\left(L[q]_{\tau}^{m}(t) -\lambda \cdot
g[q]_{\tau}^{m}(t)\right)\eta(t,q)+ \left(L[q]_{\tau}^{m}(t)
-\lambda \cdot g[q]_{\tau}^{m}(t)\right) \dot{\eta}(t,q)\right.\\
\left. +\sum_{i=0}^{m}\partial_{i+2} \left(L[q]_{\tau}^{m}(t)
-\lambda \cdot g[q]_{\tau}^{m}(t)\right)\cdot \rho^{i}(t)\right]dt
=0.
\end{multline}
Equations \eqref{eq:cnsiifm} and \eqref{eq:cnsiifm11} follow from
the fact that \eqref{mm2} holds for an arbitrary $I \subseteq
[t_1,t_2]$.
\end{proof}

\begin{definition}[Isoperimetric Higher-order constant of motion/isoperimetric conservation law with time delay]
\label{def:leicond2} A quantity
\begin{multline*}
C\{q\}_{\tau}^{m}(t) := C\Bigl(t,t+\tau,q(t),\dot{q}(t),
\ldots,q^{(m)}(t),q(t-\tau),\dot{q}(t-\tau),
\ldots,q^{(m)}(t-\tau),\\q(t+\tau),\dot{q}(t+\tau),
\ldots,q^{(m)}(t+\tau)\Bigr)
\end{multline*}
is a higher-order constant of motion with time delay $\tau$ if
\begin{equation}
\label{eq:conslaw:td} \frac{d}{dt} C\{q\}_{\tau}^{m}(t) = 0,
\end{equation}
$t\in [t_1,t_2]$, along any $q(\cdot)\in
\mathbb{W}^{m,\infty}\left([t_1-\tau,t_2], \mathbb{R}^{n}\right)$
satisfying both Theorem~\ref{Thm:ELdeordm} and
Theorem~\ref{theo:cDRifm}. The equality \eqref{eq:conslaw:td} is
then said to be a higher-order conservation law with time delay.
\end{definition}

\begin{theorem}[Isoperimetric higher-order Noether's symmetry theorem with time delay]
\label{thm:Noether} If functional \eqref{agp1} is invariant up to
the gauge-term $\Phi$ in the sense of Definition~\ref{def:invaifm}
such that satisfy condition \eqref{CDUR1}, then the quantity
$C\{q\}_{\tau}^{m}(t)$ defined by
\begin{equation}
\label{eq:TeNetm} \sum_{j=1}^{m}\psi^{j}_1\cdot
\rho^{j-1}(t)+\left(F[q]^m_{\tau}(t) -\sum_{j=1}^{m}\psi^{j}_1\cdot
q^{(j)}(t)\right)\eta(t,q) -\Phi[q]^m_{\tau}(t)
\end{equation}
for $t_1\leq t\leq t_{2}-\tau$ and by
\begin{equation}
\label{eq:TeNetm1} \sum_{j=1}^{m}\psi^{j}_2\cdot
\rho^{j-1}(t)+\left(F[q]^m_{\tau}(t) -\sum_{j=1}^{m}\psi^{j}_2\cdot
q^{(j)}(t)\right)\eta(t,q) -\Phi[q]^m_{\tau}(t)
\end{equation}
for $t_2-\tau\leq t\leq t_{2}$, is a higher-order constant of motion
with time delay (\textrm{cf.} Definition~\ref{def:leicond2})
satisfying the hypothesis \eqref{CDUR1}, where $\psi^{j}_1$ and
$\psi^{j}_2$ are given by \eqref{eq:eqprin} and \eqref{eq:eqprin11},
respectively.
\end{theorem}

\begin{proof}
We prove the theorem in the interval $t_1\leq t\leq t_{2}-\tau$. The
proof is similar in the interval $t_2-\tau\leq t\leq t_{2}$.
Equation \eqref{eq:TeNetm} follows by direct calculations:
\begin{equation}
\label{eq:TeNetm2}
\begin{split}
0&=\int_{t_1}^{t_2-\tau}\frac{d}{dt}\left[\psi_1^1\cdot\rho^0+\sum_{j=2}^{m}\psi^{j}_1
\cdot \rho^{j-1}(t)\right.\\
&\left. +\left(L[q]_{\tau}^{m}(t) -\lambda \cdot g[q]_{\tau}^{m}(t)
-\sum_{j=1}^{m}\psi^{j}_1\cdot q^{(j)}(t)\right)\eta(t,q)
-\Phi[q]^m_{\tau}(t)\right]dt\\
&=\int_{t_1}^{t_2-\tau}\left[-\dot{\Phi}[q]^m_{\tau}(t)+\rho^0(t)\cdot\frac{d}{dt}\psi^1_1
+\psi^1_1\cdot\frac{d}{dt}\rho^0(t)
+\sum_{j=2}^{m}\left(\rho^{j-1}(t)\cdot\frac{d}{dt}\psi^{j}_1
+\psi^{j}_1\cdot\frac{d}{dt}\rho^{j-1}(t)\right)\right.\\
& \qquad \left.+\eta(t,q)\frac{d}{dt}\left(L[q]_{\tau}^{m}(t)
-\lambda \cdot g[q]_{\tau}^{m}(t)-\sum_{j=1}^{m}\psi_1^{j} \cdot
q^{(j)}(t)\right)\right.\\
&\left.+\left(L[q]_{\tau}^{m}(t) -\lambda \cdot
g[q]_{\tau}^{m}(t)-\sum_{j=1}^{m}\psi_1^{j} \cdot
q^{(j)}(t)\right)\dot{\eta}(t,q)\right]dt.
\end{split}
\end{equation}
Using the Euler--Lagrange equation \eqref{eq:ELdeordm}, the
DuBois--Reymond condition \eqref{eq:DBRordm}, and relations
\eqref{eq:eqprin1}  and \eqref{eq:cnsiifm1} into \eqref{eq:TeNetm2},
we obtain:
\begin{multline}
\label{eq:dems}
\int_{t_1}^{t_2-\tau}\left[-\dot{\Phi}[q]^m_{\tau}(t)
+\left(\partial_{2} \left(L[q]_{\tau}^{m}(t) -\lambda \cdot
g[q]_{\tau}^{m}(t)\right)\right.\right.\\
\left.\left.+\partial_{m+3} \left(L[q]_{\tau}^{m}(t+\tau) -\lambda
\cdot g[q]_{\tau}^{m}(t+\tau)\right)\right)\cdot\xi(t,q)+\psi^{1}_1
\cdot(\rho^1(t)+\dot{q}(t)\dot{\tau}(t,q))\right. \\
\left. +\sum_{j=2}^{m}\left[\left(\partial_{j+1}
\left(L[q]_{\tau}^{m}(t) -\lambda \cdot
g[q]_{\tau}^{m}(t)\right)+\partial_{j+m+2}
\left(L[q]_{\tau}^{m}(t+\tau) -\lambda \cdot
g[q]_{\tau}^{m}(t+\tau)\right)-\psi_1^{j-1}\right)\cdot\rho^{j-1}(t)\right.\right.\\
\left.\left.+ \psi_1^{j}\cdot\left(\rho^{j}(t)+q^{(j)}(t)
\dot{\tau}(t,q)\right)\right]\right.\\ \left. +\partial_{1}
\left(L[q]_{\tau}^{m}(t) -\lambda \cdot
g[q]_{\tau}^{m}(t)\right)\eta(t,q) +\left(L[q]_{\tau}^{m}(t)
-\lambda \cdot g[q]_{\tau}^{m}(t)-\sum_{j=1}^{m}\psi_1^{j}
\cdot q^{(j)}(t)\right)\dot{\eta}(t,q)\right]dt\\
=\int_{t_1}^{t_2-\tau}\Bigl[
\partial_{1} \left(L[q]_{\tau}^{m}(t) -\lambda \cdot
g[q]_{\tau}^{m}(t)\right)\eta(t,q)+\left(L[q]_{\tau}^{m}(t) -\lambda
\cdot g[q]_{\tau}^{m}(t)\right)\dot{\eta}(t,q)\\ +\left(\partial_{2}
\left(L[q]_{\tau}^{m}(t) -\lambda \cdot
g[q]_{\tau}^{m}(t)\right)+\partial_{m+3}
\left(L[q]_{\tau}^{m}(t+\tau) -\lambda \cdot
g[q]_{\tau}^{m}(t+\tau)\right)\right)\cdot\xi(t,q)\\
+\psi_1^{1}
\cdot(\rho^1(t)+\dot{q}(t)\dot{\eta}(t,q))-\psi_1^1\cdot\rho^1(t)
-\psi_1^1\cdot\dot{q}(t)\dot{\eta}(t,q)+\psi_1^m\cdot\rho^m(t)\\
+\sum_{j=2}^{m}\left(\partial_{j+1} \left(L[q]_{\tau}^{m}(t)
-\lambda \cdot g[q]_{\tau}^{m}(t)\right)+\partial_{j+m+2}
\left(L[q]_{\tau}^{m}(t+\tau) -\lambda \cdot
g[q]_{\tau}^{m}(t+\tau)\right)\right]\cdot\rho^{j-1}(t)\\
-\dot{\Phi}[q]^m_{\tau}(t)\Bigr]dt
= 0.
\end{multline}
Simplification of \eqref{eq:dems} leads to the necessary condition
of invariance \eqref{eq:cnsiifm}.
\end{proof}


\section{Noether's theorem for isoperimetric problems of the optimal control with time delay}
\label{sec:MRHO1} Using Theorem~\ref{theo:tnnd}, we obtain here a
Noether's theorem for the  isoperimetric optimal control problems
with time delay introduced in \cite{GFML2016}: to minimize
\begin{equation}
\label{Pond} \mathcal{I}^{\tau}[q(\cdot),u(\cdot)] =
\int_{t_1}^{t_2} L\left(t,q(t),u(t),q(t-\tau),u(t-\tau)\right) dt
\end{equation}
subject to the delayed control system
\begin{equation}
 \label{ci}
\dot{q}(t)=\varphi\left(t,q(t),u(t),q(t-\tau),u(t-\tau)\right)\,,
\end{equation}

isoperimetric equality constraints
\begin{equation}
\label{CT2} \int_{t_1}^{t_2}
g\left(t,q(t),u(t),q(t-\tau),u(t-\tau)\right) dt=l,\,\,l\in
\mathbb{R}^k
\end{equation}
and initial condition
\begin{equation}
\label{ic} q(t)=\delta(t),\quad t\in[t_{1}-\tau,t_{1}],
\end{equation}
where $q(\cdot) \in C^{1}\left([t_1-\tau,t_2],
\mathbb{R}^{n}\right)$, $u(\cdot) \in C^{0}\left([t_1-\tau,t_2],
\mathbb{R}^{m}\right)$, the Lagrangian $L : [t_1,t_2] \times
\mathbb{R}^{n}\times\mathbb{R}^{m} \times
\mathbb{R}^{n}\times\mathbb{R}^{m} \rightarrow \mathbb{R}$ and the
velocity vector $\varphi : [t_1,t_2] \times \mathbb{R}^{n}\times
\mathbb{R}^{m} \times \mathbb{R}^{n} \times \mathbb{R}^{m}
\rightarrow \mathbb{R}^n$ are assumed to be $C^{1}$-functions with
respect to all their arguments, $t_{1} < t_{2}$ are fixed in
$\mathbb{R}$, and $\tau$ is a given positive real number such that
$\tau<t_{2}-t_{1}$. As before, we assume that $\delta$ is a given
piecewise smooth function.

\begin{remark}
In the particular case when $\varphi(t,q,u,q_\tau,u_\tau) = u$,
problem \eqref{Pond}--\eqref{ic} is reduced to the
Problem~\ref{Pb1}. The problems of the calculus of variations with
higher-order derivatives are also easily written in the optimal
control form \eqref{Pond}--\eqref{ic}. For example, the delayed
isoperimetric problem of the calculus of variations with derivatives
of second order,
\begin{equation}
\label{eq:pcvo2} I[q(\cdot)] = \int_{t_1}^{t_2}
F\left(t,q(t),\dot{q}(t),\ddot{q}(t),q(t-\tau),\dot{q}(t-\tau),\ddot{q}(t-\tau)\right)
\longrightarrow \min \, ,
\end{equation}
 is equivalent to problem
\begin{gather*}
I[q^0(\cdot),q^1(\cdot),u(\cdot)]
= \int_{t_1}^{t_2} F\left(t,q^0(t),q^1(t),u(t),q^0(t-\tau),q^1(t-\tau),u(t-\tau)\right) \longrightarrow \min \, , \\
\begin{cases}
\dot{q}^0(t) = q^1(t) \, , \\
\dot{q}^1(t) = u(t) \, ,\\
\dot{q}^0(t-\tau) = q^1(t-\tau) \, , \\
\dot{q}^1(t-\tau) = u(t-\tau) \, ,
\end{cases}
\end{gather*}
where $F=L-\lambda\cdot g$ and
$g=g\left(t,q(t),\dot{q}(t),\ddot{q}(t),q(t-\tau),\dot{q}(t-\tau),\ddot{q}(t-\tau)\right)\,.$
\end{remark}

\begin{notation}
We introduce the operators $[\cdot,\cdot]_{\tau}$ and
$[\cdot,\cdot,\cdot]_{\tau}$ defined by
\begin{equation*}
[q,u]_{\tau}(t)=\left(t,q(t),u(t),q(t-\tau),u(t-\tau)\right),
\end{equation*}
where $q(\cdot) \in C^{1}\left([t_1-\tau,t_2],
\mathbb{R}^{n}\right)$ and $u(\cdot) \in C^{0}\left([t_1-\tau,t_2],
\mathbb{R}^{m}\right)$; and
\begin{equation*}
[q,u,p,\lambda]_{\tau}(t)=\left(t,q(t),u(t),q(t-\tau),u(t-\tau),p(t),\lambda\right),
\end{equation*}
where $q(\cdot) \in C^{1}\left([t_1-\tau,t_2],
\mathbb{R}^{n}\right)$, $p(\cdot) \in C^{1}\left([t_1,t_2],
\mathbb{R}^{n}\right)$, $u(\cdot) \in C^{0}\left([t_1-\tau,t_2],
\mathbb{R}^{m}\right)$ and $\lambda\in\mathbb{R}^k$.
\end{notation}

\begin{definition}
The delayed differential control system \eqref{ci} is called an
\emph{isoperimetric control system with time delay}.
\end{definition}

\begin{definition}(Isoperimetric process with time delay)
An admissible pair $(q(\cdot),u(\cdot))$ that satisfies the
isoperimetric control system \eqref{ci} and the isoperimetric
constraints \eqref{CT2} is said to be a \emph{isoperimetric process
with time delay}.
\end{definition}

\begin{theorem}
\label{theo:pmpnd}(Isoperimetric Pontryagin maximum principle
\cite{GFML2016}) If $(q(\cdot),u(\cdot))$ is a minimizer of
\eqref{Pond}--\eqref{ic}, then there exists a covector function
$p(\cdot)\in C^{1}\left([t_1,t_2], \mathbb{R}^{n}\right)$ such that
for all $t\in[t_1-\tau,t_2]$ the following conditions hold:
\begin{itemize}
\item \emph{the isoperimetric Hamiltonian systems with time delay}
\begin{equation}
\label{eq:Hamnd}
\begin{cases}
\dot{q}(t)=\partial_6 H[q,u,p,\lambda]_{\tau}(t)\\
\dot{p}(t)=-\partial_2 H[q,u,p,\lambda]_{\tau}(t) -\partial_4
H[q,u,p,\lambda]_{\tau}(t+\tau)
\end{cases}
\end{equation}
for $t_{1}\leq t\leq t_{2}-\tau$, and
\begin{equation}
\label{eq:Hamnd2}
\begin{cases}
\dot{q}(t) = \partial_6 H[q,u,p,\lambda]_{\tau}(t)  \\
\dot{p}(t) = -\partial_2 H[q,u,p,\lambda]_{\tau}(t)
\end{cases}
\end{equation}
for $t_{2}-\tau\leq t\leq t_{2}$;

\item \emph{the isoperimetric stationary conditions with time delay}
\begin{equation}
\label{eq:CE}
\partial_3 H[q,u,p,\lambda]_{\tau}(t)+\partial_5 H[q,u,p,\lambda]_{\tau}(t+\tau)= 0
\end{equation}
for $t_{1}\leq t\leq t_{2}-\tau$, and
\begin{equation}
\label{eq:CE12}
 \partial_3 H[q,u,p]_{\tau}(t)= 0
\end{equation}
for $t_{2}-\tau\leq t\leq t_{2}$;
\end{itemize}
where the isoperimetric Hamiltonian $H$ is defined by
\begin{equation}
\label{eq:Hnd11} H[q,u,p,\lambda]_{\tau}(t) = L[q,u]_{\tau}(t)
-\lambda\cdot g[q,u]_{\tau}(t)+ p(t) \cdot \varphi[q,u]_{\tau}(t).
\end{equation}
\end{theorem}

\begin{definition}
\label{scale:Pont:Ext} A triplet
$\left(q(\cdot),u(\cdot),p(\cdot)\right)$ satisfying the conditions
of Theorem~\ref{theo:pmpnd} is called an \emph{isopereimetric
Pontryagin extremal with time delay}.
\end{definition}

We define the notion of invariance for Problem
\eqref{Pond}--\eqref{ic} in terms of the Hamiltonian, by introducing
the augmented functional as in [3]:

\begin{equation}
\label{eq:pcond} \mathcal{J}[q(\cdot),u(\cdot),p(\cdot)] =
\int_{t_{1}}^{t_{2}}
\left[H(t,q(t),u(t),q(t-\tau),u(t-\tau),p(t))-p(t)\cdot\dot{q}(t)\right]dt
\end{equation}
subject to \eqref{ic}, where $H$ is given by \eqref{eq:Hnd11}. The
notion of invariance for \eqref{Pond}--\eqref{ci} is defined using
the invariance of \eqref{eq:pcond}.

\begin{definition}[\textrm{cf.} Definition~\ref{def:invndLIP}]
\label{def:invnd-co1} Consider the following $s$-parameter group of
infinitesimal transformations:
\begin{equation}
\label{eq:tinfnd}
\begin{cases}
\bar{t} = t + s\eta(t,q,u) + o(s),\\
\bar{q}(t) = q(t) + s\xi(t,q,u) + o(s),\\
\bar{u}(t) = u(t) + s\varrho(t,q,u) + o(s),\\
\bar{p}(t) = p(t) +s\varsigma(t,q,u) + o(s),
\end{cases}
\end{equation}
where $\eta \in C^1\left(\mathbb{R}^{1+n+m}, \mathbb{R}\right)$,
$\xi, \varsigma \in C^1\left(\mathbb{R}^{1+n+m},
\mathbb{R}^n\right)$, $\varrho \in C^0\left(\mathbb{R}^{1+n+m},
\mathbb{R}^m\right)$ are given functions. The functional
\eqref{eq:pcond} is said to be invariant under \eqref{eq:tinfnd} if
\begin{multline*}
\left.\frac{d}{ds}\right|_{s=0} \int_{\bar{t}(I)}
\Biggl[H\Biggl(t+s\eta,q+s\xi,u+s\varrho,q(t-\tau)+s\xi_{\tau}(t),u(t-\tau)+s\varrho_{\tau}(t),\\
-(p+s\varsigma)\cdot\frac{\dot{q}+s\dot{\xi}}{1+s\dot{\eta}}\Biggr)
(1+s\dot{\eta})\Biggr] dt = 0
\end{multline*}
for any  subinterval $I \subseteq [t_1,t_2]$, where
$\varrho_{\tau}(t)=\varrho(t-\tau,q(t-\tau),u(t-\tau))$ and\\
$\xi_{\tau}(t)=\xi\left(t-\tau,q(t-\tau),u(t-\tau)\right)$.
\end{definition}

\begin{definition}
\label{lco} A quantity $C(t,q(t),q(t-\tau),u(t),u(t-\tau),p(t))$,
constant for $t \in [t_1,t_2]$ along any isoperimetric Pontryagin
extremal with delay $(q(\cdot),u(\cdot),p(\cdot))$ of problem
\eqref{Pond}--\eqref{ic}, is said to be an \emph{isoperimetric
constant of motion with delay} for \eqref{Pond}--\eqref{ic}.
\end{definition}

Theorem~\ref{thm:NT:OC} gives a Noether-type theorem for
isoperimetric optimal control problems with time delay.

\begin{theorem}[Isoperimetric Noether symmetry theorem with time delay in Hamiltonian form]
\label{thm:NT:OC} If we have invariance in the sense of
Definition~\ref{def:invnd-co1}, then
\begin{multline}
\label{eq:tnnd-co} C(t,q(t),q(t-\tau),u(t),u(t-\tau),p(t))
= -p(t)\cdot\xi\left(t,q(t),u(t)\right)\\
+ H\left(t,q(t),u(t),q(t-\tau),u(t-\tau),p(t)\right)
\eta\left(t,q(t),u(t)\right)
\end{multline}
is an isoperimetric constant of motion with delay (\textrm{cf.}
Definition~\ref{lco}) for \eqref{Pond}--\eqref{ic}.
\end{theorem}

\begin{proof}
The constant of motion with delay \eqref{eq:tnnd-co} is obtained by
applying Theorem~\ref{theo:tnnd} to problem \eqref{eq:pcond}.
\end{proof}

\begin{remark}
The constant of motion with time delay \eqref{eq:tnnd-co} has the
same expression in the two intervals $t_{1}\leq t\leq t_{2}-\tau$
and $t_{2}-\tau\leq t\leq t_{2}$.
\end{remark}

\begin{remark}
For the isoperimetric problem of the calculus of variations
\eqref{Pe}--\eqref{Pe2}, the Hamiltonian \eqref{eq:Hnd11} takes the
form $H = L + p \cdot u$, with $u = \dot{q}$ and $p(t) = -\partial_3
L[q]_\tau(t)-\partial_5 L[q]_\tau(t+\tau)$. In this case the
isoperimetric constant of motion with delay \eqref{eq:tnnd-co}
reduces to \eqref{eq:tnnd} in the interval $t_{1}\leq t\leq
t_{2}-\tau$ and to \eqref{eq:Noeth} in the interval $t_{2}-\tau\leq
t\leq t_{2}$ with $\Phi\equiv 0\,.$
\end{remark}

 \begin{corollary}(Isoperimetric Noether's theorem
 for problems of the calculus of variations with second-order
 derivatives) \label{eq:H142}
 For the second-order problem of the calculus of variations
\eqref{eq:pcvo2}, the isoperimetric constant of motion with delay
\eqref{eq:tnnd-co} is equivalent to

\begin{multline}
\label{eq:H14} F[q]^2_{\tau}(t)\tau+\left(
\partial_3F[q]^2_{\tau}(t)+\partial_6F[q]^2_{\tau}(t+\tau)
-\frac{d}{dt}\left(\partial_4F[q]^2_{\tau}(t)+\partial_7F[q]^2_{\tau}(t+\tau)\right)\right)
\cdot (\xi_{0}-\dot{q}\tau)\\
+\left(\partial_4F[q]^2_{\tau}(t)+\partial_7F[q]^2_{\tau}(t+\tau)\right)
\cdot (\xi_{1}-\ddot{q}\tau)
\end{multline}
for $t_{1}\leq t\leq t_{2}-\tau$, and
\begin{equation}
\label{eq:H141} F[q]^2_{\tau}(t)\tau+\left(
\partial_3F[q]^2_{\tau}(t)
-\frac{d}{dt}\partial_4F[q]^2_{\tau}(t)\right) \cdot
(\xi_{0}-\dot{q}\tau) +\partial_4F[q]^2_{\tau}(t) \cdot
(\xi_{1}-\ddot{q}\tau)
\end{equation}
for $t_{2}-\tau\leq t\leq t_{2}$
 \end{corollary}
\begin{proof} For simplicity, we only prove the corollary in the
interval $t_{2}-\tau\leq t\leq t_{2}\,.$ For the problem of the
calculus of variations with second-order derivatives, one has
\begin{gather*}
H\left(t,q^{0}(t),q^{1}(t),u(t),q^{0}(t-\tau),q^{1}(t-\tau),u(t-\tau),p^{0},p^{1}\right)\\=
F(t,q^{0},q^{1},u,q^{0}(t-\tau),q^{1}(t-\tau),u(t-\tau))+p^{0}
q^{1}+p^{1}u\,,
\end{gather*}
and
\begin{equation*}
\begin{cases}
q^{0}(t)=q(t)\\
q^{1}(t)=\dot{q}(t)\\
u(t)=\ddot{q}(t)\\
q^{0}(t-\tau)=q(t-\tau)\\
q^{1}(t-\tau)=\dot{q}(t-\tau)\\
u(t-\tau)=\ddot{q}(t-\tau)
\end{cases}
\end{equation*}

 Using these equalities, it follows from the isoperimetric
Pontryagin Maximum Principle (Theorem~\ref{theo:pmpnd}) that
\begin{gather*}
\frac{\partial H}{\partial u}=0\Leftrightarrow
{p}^{1}=\frac{\partial F}{\partial \ddot{q}} \, ,\\
\dot{p}^{0}=-\frac{\partial H}{\partial q^{0}}=\frac{\partial
F}{\partial q} \, ,\\
\dot{p}^{1}=-\frac{\partial H}{\partial q^{1}} \Leftrightarrow
p^{0}=\frac{\partial F}{\partial \dot{q}}-\frac{d}{dt}\frac{\partial
F}{\partial \ddot{q}} \, .
\end{gather*}
In this case, the constant of motion \eqref{eq:tnnd-co} takes the
form
\begin{equation*}
C={\mathcal{H} }\tau-p^{0} \cdot \xi_{0}-p^{1} \cdot \xi_{1} \, ,
\end{equation*}
and substituting $H$, $p^{0}$ and $p^{1}$ by its expressions, the
intended result is obtained.
\end{proof}

\begin{remark} We can easily verify that for $m=2$ the quantities
\eqref{eq:H14} and \eqref{eq:H141} coincide with \eqref{eq:TeNetm}
and \eqref{eq:TeNetm1}, respectively, in the case $\Phi\equiv 0\,.$
\end{remark}


\section{Illustrative examples}\label{exe}

In this section we consider two examples where the isoperimetric
problems do not depend explicitly on the independent variable $t$
(autonomous case).

\begin{example}
Consider the second-order isoperimetric problem of the calculus of
variations with time delay
\begin{equation}
\label{eq:ex}
\begin{gathered}
J^{2}_1[q(\cdot)]=\int_0^2\left(\ddot{q}(t)+\ddot{q}(t-1)\right)^2dt \longrightarrow \min,\\
q(t)=-t^4 \, ,~-1\leq t\leq 0,\quad \dot{q}(2)=-32, \quad q(2)=-14,
\end{gathered}
\end{equation}
subject to  isoperimetric equality constraints
\begin{equation}
\label{CT5} I^{1}[q(\cdot)]=\int_{0}^{2}
\left(\dot{q}+\dot{q}(t-1)\right)^2dt=l
\end{equation}
in the class of functions $q(\cdot)\in
Lip\left([-1,2];\mathbb{R}\right)$. For this example, the augmented
Lagrangian $F$ is given as
\begin{equation}\label{}
    F=\left(\ddot{q}(t)+\ddot{q}(t-\tau)\right)^2
    -\lambda\left(\dot{q}+\dot{q}(t-1)\right)^2 \,.
\end{equation}
From Corollary~\ref{cor:16} with $m=2$, one obtains that any
solution to problem \eqref{eq:ex}-\eqref{CT5} must satisfy
\begin{equation}
\label{eq:ex:EL1}
2q^{(iv)}(t)+q^{(iv)}(t-1)+q^{(iv)}(t+1)+2\lambda\left(\ddot{q}(t)+\ddot{q}(t-1)+\ddot{q}(t+1)\right)=0
, \quad 0\leq t\leq 1,
\end{equation}
\begin{equation}
\label{eq:ex:EL2}
q^{(iv)}(t)+q^{(iv)}(t-1)+\lambda\left(\ddot{q}(t)+\ddot{q}(t-1)\right)=0,
\quad 1\leq t\leq 2\,.
\end{equation}
 Because problem
\eqref{eq:ex}--\eqref{CT5} is autonomous, we have invariance, in the
sense of Definition~\ref{def:invaifm}, with $\eta\equiv 1$ and
$\xi\equiv 0$. Simple calculations show that isoperimetric Noether's
constant of motion with time delay \eqref{eq:H14}--\eqref{eq:H141}
coincides with the DuBois--Reymond condition
\eqref{eq:DBRordm}--\eqref{eq:DBRordm:2} with $m=2$:
\begin{multline}
\label{eq:ex:DBR1} \left(\ddot{q}(t)+\ddot{q}(t-\tau)\right)^2
    -\lambda\left(\dot{q}+\dot{q}(t-1)\right)^2\\
    +2\dot{q}(t)\left[\lambda\left(2\dot{q}(t)+\dot{q}(t-1)+\dot{q}(t+1)\right)
    +2q^{(iii)}(t)+q^{(iii)}(t-1)+q^{(iii)}(t+1)\right]\\
    -2\ddot{q}(t)\left(2\ddot{q}(t)+\ddot{q}(t-1)+\ddot{q}(t+1)\right)=c_1,\quad
0\leq t\leq 1,
\end{multline}

\begin{multline}\label{eq:ex:DBR22}
\left(\ddot{q}(t)+\ddot{q}(t-\tau)\right)^2
    -\lambda\left(\dot{q}+\dot{q}(t-1)\right)^2
    +2\dot{q}(t)\left[\lambda\left(\dot{q}(t)+\dot{q}(t-1))\right)
    +q^{(iii)}(t)+q^{(iii)}(t-1)\right]\\
    -2\ddot{q}(t)\left(\ddot{q}(t)+\ddot{q}(t-1)\right)=c_2,\quad 1\leq t\leq 2,
\end{multline}
where $c_1$ and $c_2$ are constants.

One can easily check that function  $q(\cdot)\in
Lip\left([-1,2];\mathbb{R}^n\right)$ defined by
\begin{equation}
\label{eq:ext:ex:22b} q(t)=
\begin{cases}
-t^4 & ~\textnormal{for}~ -1< t\leq 0\\
t^4 & ~\textnormal{for}~ 0< t\leq 1\\
-t^4+2 & ~\textnormal{for}~ 1< t\leq 2
\end{cases}
\end{equation}
is an isoperimetric Euler--Lagrange extremal, i.e., satisfies
\eqref{eq:ex:EL1}--\eqref{eq:ex:EL2} and is also a isoperimetric
DuBois--Reymond extremal, i.e., satisfies
\eqref{eq:ex:DBR1}--\eqref{eq:ex:DBR22}. Corollary~\ref{eq:H142}
asserts the validity of Noether's constant of motion, which is here
 verified: \eqref{eq:H14}--\eqref{eq:H141} holds along
\eqref{eq:ext:ex:22b} with  $\eta\equiv 1$, and $\xi\equiv 0$.

\end{example}

\begin{example}
\label{ex:2} Let us consider an isoperimetric autonomous optimal
control problem with time delay, \textrm{i.e.}, the situation when
$L$, $\varphi$ and $g$ in \eqref{Pond}--\eqref{CT2} do not depend
explicitly on $t$. In this case one has invariance, in the sense of
Definition~\ref{def:invnd-co1}, for $\eta \equiv 1$ and $\xi =
\varrho = \varsigma \equiv 0$. It follows from
Theorem~\ref{thm:NT:OC} that
\begin{equation}
\label{ceg} H\left(q(t),u(t),q(t-\tau),u(t-\tau),p(t)\right) =
\text{constant}
\end{equation}
along any isoperimetric Pontryagin extremal with delay
$(q(\cdot),u(\cdot),p(\cdot))$ of the problem. In the language of
mechanics \eqref{ceg} is called  conservation of energy.
\end{example}



\end{document}